\documentclass{article}

\usepackage{amssymb}
\usepackage{amsthm}
\usepackage{amscd}
\usepackage{amsmath}
\theoremstyle{plain}
\newtheorem{theorem}{Theorem}[section]
\newtheorem*{theorem*}{Theorem}

\newtheorem*{maintheorem*}{Main Theorem}
\newtheorem{prop}[theorem]{Proposition}
\newtheorem{corollary}[theorem]{Corollary}
\newtheorem{lemma}[theorem]{Lemma}

\newtheorem*{conjecture*}{Conjecture}

\newtheorem{question}[theorem]{Question}
\theoremstyle{definition}
\newtheorem{definition}[theorem]{Definition}
\newtheorem*{definition*}{Definition}
\newtheorem*{example*}{Example}

\newtheorem*{notation*}{Notation}
\newtheorem*{notation-conv*}{Notation and convention}
\newtheorem*{convention*}{Convention}
\newtheorem{remark}[theorem]{Remark}

\newcommand{\Z}{{\mathbb Z}}
\newcommand{\C}{{\mathbb C}}
\newcommand{\Q}{{\mathbb Q}}
\newcommand{\R}{{\mathbb R}}

\newcommand{\N}{{\mathbb N}}

\begin{document}
\title{ Twisted Fourier-Mukai number of a $K3$ surface }
\author{Shouhei Ma \\
{\small Graduate School of Mathematical Sciences, University of Tokyo   } \\
{\small 3-8-1\: Komaba\: Meguro-ku, Tokyo\: 153-8914, Japan } \\
{\small E-mail: sma@ms.u-tokyo.ac.jp } }
\date{}
\maketitle

\begin{abstract}
We give a counting formula for 
the twisted Fourier-Mukai partners of 
a projective $K3$ surface. 
As an application, 
we describe   
all twisted Fourier-Mukai partners 
of a projective $K3$ surface 
of Picard number $1$. 
\end{abstract}

\section{Introduction}

Fourier-Mukai (FM) partners of a projective $K3$ surface $S$ 
have been studied since Mukai's work \cite{Mu}. 
The basic results due to 
Mukai and Orlov (\cite{Or}) are as follows :  
\begin{itemize}
 \item[(1)] If a $2$-dimensional moduli space $M$ of stable sheaves on $S$ 
            is non-empty and compact, 
            then $M$ is a $K3$ surface. 
            When $M$ is  {\it fine\/}, 
            a universal sheaf induces an equivalence 
            $D^{b}(M) \simeq D^{b}(S)$. 
 \item[(2)] Every FM partner of $S$ is isomorphic to a certain $2$-dimensional 
            {\it fine\/} moduli space of stable sheaves on $S$. 
 \item[(3)] A projective $K3$ surface $S'$ is a FM partner of $S$ 
            if and only if their Mukai lattices are Hodge isometric. 
\end{itemize} 
Using Mukai-Orlov's theorem (especially $(3)$),          
Hosono-Lian-Oguiso-Yau (\cite{H-L-O-Y}) 
derived a counting formula for the FM partners of $S$.

On the other hand, 
it has been recognized that 
studying {\it twisted\/} FM partners 
as well as usual (i.e, untwisted) FM partners of $S$ 
is important 
when analyzing the derived category $D^{b}(S)$ of $S$. 
Recall that 
a twisted FM partner of $S$ is 
a twisted $K3$ surface $(S', \alpha ')$ such that 
there is an equivalence  
$D^{b}(S', \alpha ') \simeq D^{b}(S)$. 
C\u ald\u araru (\cite{Ca1}, \cite{Ca2}), 
Huybrechts-Stellari (\cite{H-S1}, \cite{H-S2}), 
and Yoshioka (\cite{Yo}) 
generalized Mukai and Orlov's results to twisted situation. 
When $S$ is untwisted, 
their results can be stated as follows : 
\begin{itemize}
 \item[(1)] When a $2$-dimensional {\it coarse\/} moduli space $M$ of stable sheaves on $S$ 
            is a $K3$ surface, 
            a {\it twisted\/} universal sheaf induces an equivalence 
            $D^{b}(M, \alpha ) \simeq D^{b}(S)$,  
            where $\alpha $ is the obstruction to the existence of a universal sheaf. 
 \item[(2)] Every twisted FM partner of $S$ is isomorphic to a certain $2$-dimensional 
            {\it coarse\/} moduli space of stable sheaves on $S$ endowed with a natural twisting. 
 \item[(3)] A twisted $K3$ surface $(S', \alpha ')$ is a twisted FM partner of $S$ 
            if and only if the twisted Mukai lattice of $(S', \alpha ')$ 
            and the Mukai lattice of $S$ are Hodge isometric. 
\end{itemize}

Let 
${\rm FM\/}^{d}(S)$ 
be the set of 
isomorphism classes of twisted FM partners $(S', \alpha ')$ of $S$ 
with 
${\rm ord\/}(\alpha ')=d$. 
The number 
$\#{\rm FM\/}^{d}(S)$ 
is an invariant of the category $D^{b}(S)$ 
representing the number of certain geometric origins of $D^{b}(S)$. 
The aim of this paper 
is to present an explicit formula for 
$\#{\rm FM\/}^{d}(S)$. 
It enables us to calculate 
$\#{\rm FM\/}^{d}(S)$ 
from 
lattice-theoretic informations about the Neron-Severi lattice $NS(S)$ 
and 
the knowledge of the group $O_{Hodge}(T(S))$ 
of the Hodge isometries of the transcendental lattice $T(S)$.

Let 
$I^{d}(D_{NS(S)})$ 
be the set of 
${\rm order\/}=d$ isotropic elements of the discriminant form $D_{NS(S)}$. 
From an element $x\in I^{d}(D_{NS(S)})$ 
we can construct overlattices 
$M_{x}$, $T_{x}$ of $NS(S)$, $T(S)$ respectively, 
and a homomorphism 
$\alpha _{x} : T_{x} \twoheadrightarrow {\Z}/d{\Z}$. 
After defining certain subsets 
$\mathcal{G}_{1}(M)$, $\mathcal{G}_{2}(M)$ of 
the genus $\mathcal{G}(M)$ of a lattice $M$ 
such that 
$\mathcal{G}(M) = \mathcal{G}_{1}(M) \sqcup \mathcal{G}_{2}(M)$, 
our formula is stated as follows.

\begin{theorem}[Theorem \ref{the formula}]\label{main1} 
For a projective $K3$ surface $S$ 
the following formula holds. 
\begin{eqnarray*}
\#{\rm FM\/}^{d}(S) & = &
 \mathop{\sum}_{x}   
   \left\{ \: \sum_{M} \# \Bigl( \: O_{Hodge}(T_{x}, \alpha _{x}) \backslash O(D_{M}) / O(M) \: \Bigr) \right. \\
&    & \: \: \: \: \: \: \: 
 \left.   
+ \; \varepsilon (d) \: \sum_{M'} \# \Bigl( \: O_{Hodge}(T_{x}, \alpha _{x}) \backslash O(D_{M'}) / O(M') \: \Bigr) \: \right\}
\end{eqnarray*} 
Here 
$x$ runs over the set $O_{Hodge}(T(S)) \backslash I^{d}(D_{NS(S)})$ 
and the lattices $M$, $M'$ run over the sets 
$\mathcal{G}_{1}(M_{x})$, $\mathcal{G}_{2}(M_{x})$ respectively. 
The natural number 
$\varepsilon (d)$ is defined by 
$\varepsilon (d)=1$ if $d=1, 2$, and 
$\varepsilon (d)=2$ if $d\geq 3$. 
\end{theorem}

When $d=1$, 
Theorem \ref{main1} is Hosono-Lian-Oguiso-Yau's formula. 
Even though Theorem \ref{main1} seems complicated in appearance, 
it is rather adequate for calculations for the following reasons :  
\begin{itemize}
   \item The genus $\mathcal{G}(M)$, 
         and the homomorphism $O(M) \to O(D_{M})$ 
         have been studied in lattice theory, 
         and much is known (cf. \cite{Ni}).
   \item If we can identify the discriminant form $D_{NS(S)}$, 
         it is easy to calculate $I^{d}(D_{NS(S)})$.
   \item The group $O_{Hodge}(T(S))$ and its subgroup $O_{Hodge}(T_{x}, \alpha _{x})$ are cyclic, 
         and the Euler function of  the order of $O_{Hodge}(T(S))$ 
         must divides ${\rm rk\/}(T(S))$ (cf. \cite{H-L-O-Y}). 
         For example, 
         $O_{Hodge}(T(S)) = \{ \pm {\rm id\/} \}$ 
         whenever ${\rm rk\/} (T(S))$ is odd. 
\end{itemize}
In fact, 
we shall derive more simple formulas for   
$\# {\rm FM\/}^{d}(S)$ 
for several classes of $K3$ surfaces : 
Jacobian $K3$ surfaces,   
in particular $K3$ surfaces with ${\rm rk\/}(NS(S)) \geq 13$, 
$K3$ surfaces with $2$-elementary $NS(S)$, 
and 
$K3$ surfaces with ${\rm rk\/}(NS(S)) = 1$.

As a by-product of the proof of Theorem \ref{main1}, 
we obtain an upper bound for the twisted FM number of 
a {\it twisted\/} $K3$ surface (Proposition \ref{the upper bound}). 
The sharpness of the estimate  
is related with a C\u ald\u araru's problem stated in \cite{Ca1} 
(Remark \ref{sharpness and a conjecture}).  
Note that 
if a twisted $K3$ surface $(S', \alpha ')$ 
has an untwisted FM partner, 
e.g, if ${\rm rk\/}(NS(S')) \geq 12$ (\cite{H-S1}), 
then 
Theorem \ref{main1} 
gives a counting formula for the twisted FM number of $(S', \alpha ')$.

As an application of Theorem \ref{main1}, 
we shall describe all twisted FM partners of a $K3$ surface of Picard number $1$ 
as follows.

\begin{theorem}[Theorem \ref{(d, n) partners}]\label{main2}
Let 
$S$ be a projective $K3$ surface with 
$NS(S)={\Z}H$, $(H, H)=2n$.  
We have 
${\rm FM\/}^{d}(S) \ne \phi $ 
if and only if 
$d^{2}|n$. 
For a natural number $d$ with $d^{2}|n$, 
let 
$\{ v_{\sigma , k} \; | \; (\sigma , k)\in \Sigma \times ({\Z}/d{\Z})^{\times} \}$ 
be the primitive isotropic 
vectors in $\widetilde{NS}(S)$ 
defined by $(\ref{eqn:Mukai vector})$.  
Let
$M_{\sigma , k}$ be 
the moduli space of Gieseker stable sheaves on $S$ with Mukai vector $v_{\sigma , k}$, 
endowed with   
the obstruction 
$\alpha _{\sigma , k} \in {\rm Br\/}(M_{\sigma , k})$ 
to the existence of a universal sheaf.

$(1)$ 
If $d^{2}<n$ or $d\leq 2$, then 
\begin{equation*} 
{\rm FM\/}^{d}(S) = \left\{ \left. ( M_{\sigma , k}, \alpha _{\sigma , k} ) \:  \right| \: 
                            (\sigma , k) \in  \Sigma \times ({\Z}/d{\Z})^{\times} \: \right\} . 
\end{equation*}

$(2)$ 
Assume that 
$d^{2}=n$ and $d\geq 3$. 
Choose a set 
$\{ j \} \subset ({\Z}/d{\Z})^{\times}$ 
of representatives of  
$({\Z}/d{\Z})^{\times} / \{ \pm {\rm id\/} \} $.  
Then 
\begin{equation*} 
{\rm FM\/}^{d}(S) = \left\{ \left. ( M_{\sigma , k}, \alpha _{\sigma , k} ) \: \right| \: 
                            (\sigma , k) \in \Sigma \times \{ j \} \: \right\} . 
\end{equation*}
\end{theorem}

Theorem \ref{main2} is a twisted generalization of a result of \cite{H-L-O-Y'}.

This paper is organized as follows. 
In Sect.2.1, 
we prepare some lattice theory 
following \cite{Ni}.
In Sect.2.2, 
we recall from \cite{Ca1} and \cite{H-S1}  
several facts about twisted $K3$ surfaces.
In Sect.3.1, 
we describe a quotient set of ${\rm FM\/}^{d}(S)$ 
in terms of lattice theory. 
In Sect.3.2 and 3.3, 
we describe the fibers of the quotient map 
in terms of lattice theory.   
In Sect.4,  
we derive Theorem \ref{main1} and its corollaries. 
In Sect.5,  
we prove Theorem \ref{main2}.

\begin{notation*}
By an {\it even lattice\/},
we mean 
a free ${\Z}$-module $L$ 
equipped with a non-degenerate symmetric bilinear form 
$( , ) : L\times L \to {\Z}$ 
satisfying 
$(x, x)\in 2{\Z}$ 
for all 
$x \in L$.
The group of isometries of $L$ is denoted by $O(L)$.
For two lattices $L$ and $M$, 
${\rm Emb\/}(L, M)$ is the set of the primitive embeddings of $L$ into $M$.
An element $l\in L$ is said to be  
{\it isotropic\/}
if $(l, l)=0$. 
The {\it hyperbolic plane\/} $U$ 
is the even indefinite unimodular lattice 
${\Z}e+{\Z}f, \: (e, e)=(f, f)=0, \: (e, f)=1$.

By a $K3$ {\it surface\/}, 
we mean a projective $K3$ surface over ${\C}$.
The Neron-Severi (resp. transcendental) lattice of a $K3$ surface $S$ 
is denoted by $NS(S)$ (resp. $T(S)$).
Set \: 
$\widetilde{NS} (S):= H^{0}(S, {\Z}) \oplus NS(S)\oplus H^{4}(S, {\Z})$, \: 
$\Lambda _{K3} := U^{3}\oplus E_{8}^{2}$, \: 
and \: 
$\widetilde{\Lambda }_{K3} := U \oplus \Lambda _{K3} = U^{4}\oplus E_{8}^{2}$.
\end{notation*}

\noindent{\em Acknowledgement.}
The author would like to 
thank Professor Shinobu Hosono 
for conversations that inspired this work. 
His thanks also go to 
Professors Keiji Oguiso and Ken-Ichi Yoshikawa 
for valuable comments. 
The idea of constructing moduli spaces in Section 5 
is due to Professor Oguiso.

\section{Preliminaries}

\subsection{Even lattices}

For an even lattice $L$, 
we can associate 
a finite Abelian group $D_{L}:= L^{\vee}/L $
equipped with a quadratic form 
$q_{L}:D_{L}\to {\Q}/2{\Z}, \: \: q_{L}(x+L)=(x, x)+2{\Z}$ \: 
for $x\in L^{\vee}$.
We call $(D_{L}, q_{L})$ 
the {\it discriminant form\/} of $L$.
There is a natural homomorphism 
$r_{L}: O(L)\to O(D_{L}, q_{L})$,  
whose kernel is denoted by $O(L)_{0}$.
We often write just 
$(D_{L}, q)$ or $D_{L}$ (resp. $r$) 
instead of $(D_{L}, q_{L})$ (resp. $r_{L}$).
Set 
\[
I^{d}(D_{L}):=\{ \: x\in D_{L} \: | \: q_{L}(x)=0 \in {\Q}/2{\Z}, \: {\rm ord\/}(x)=d \: \} .
\]

\begin{prop}[\cite{Ni}]\label{Nikulin1}
Let $M$ be a primitive sublattice of an even unimodular lattice $L$
with the orthogonal complement $M^{\perp}$. 
Then

{\rm (1)\/} There is an isometry $\lambda :(D_{M},q_{M})
\stackrel{\simeq}{\rightarrow} (D_{M^{\perp}},-q_{M^{\perp}})$.

{\rm (2)\/} For two isometries 
$\gamma _{M} \in O(M)$ and $\gamma _{M^{\perp}} \in O(M^{\perp})$, 
there is an isometry $\gamma _{L}\in O(L)$ 
such that $\gamma _{L}|_{M\oplus M^{\perp}}=\gamma _{M} \oplus \gamma _{M^{\perp}}$ 
if and only if  
$\lambda \circ r_{M}(\gamma _{M} )=r_{M^{\perp}}(\gamma _{M^{\perp}} )\circ \lambda $.
\end{prop}

Two even lattices $L$ and $L'$ are said to be $isogenus$ 
if 
$L\otimes {\Z}_{p} \simeq L'\otimes {\Z}_{p}$ 
for every prime number $p$ 
and ${\rm sign\/}(L)={\rm sign\/}(L')$. 
By \cite{Ni}, 
these are equivalent to the conditions that
$(D_{L},q)\simeq (D_{L'},q)$ and 
${\rm sign\/}(L)={\rm sign\/}(L')$. 
The set of the isometry classes of the lattices isogenus to $L$ is denoted by $\mathcal{G} (L)$.

\begin{prop}[\cite{Ni}]\label{Nikulin2}
Let $l(D_{L})$ be the minimal number of the generators of $D_{L}$.
If $L$ is indefinite and \, ${\rm rk\/}(L)\geq l(D_{L})+2$,\:  
then $\mathcal{G} (L)= \{ L \} $ 
and the homomorphism $r_{L}$ is surjective.
\end{prop}

Let $L$ be a lattice of 
${\rm sign\/}(L)=(b_{+}, b_{-})$, $b_{+}>0$. 
We denote by $\Omega _{L}$ 
the set of oriented positive-definite $b_{+}$-planes in $L\otimes {\R}$, 
which is an open subset of the oriented Grassmannian 
and has two connected components.
A choice of a component of $\Omega _{L}$, 
which is equivalent to 
a choice of an orientation for a positive-definite $b_{+}$-plane in $L\otimes {\R}$, 
is sometimes called an 
{\it orientation\/} of $L$. 
For example, 
for a $K3$ surface $S$ 
the lattice $\widetilde{NS}(S)$ equipped with the Mukai pairing 
is of ${\rm sign\/}(\widetilde{NS}(S))=(2, \rho (S))$. 
We define 
the orientation of $\widetilde{NS}(S)$ 
so that 
${\R}(1, 0, -1)\oplus {\R}(0, H, 0)$ is of positive orientation, 
where $H\in NS(S)$ is an ample class.

For a subgroup 
$\Gamma \subset O(L)$, 
let $\Gamma ^{+}$ be the subgroup of $\Gamma $ 
consisting of orientation-preserving isometries in $\Gamma $. 
The group $\Gamma ^{+}$ 
is of index at most $2$ in $\Gamma $. 
We define the subsets 
$\mathcal{G}_{1}(L), \mathcal{G}_{2}(L) \subset \mathcal{G}(L)$ 
by 
\begin{eqnarray*}
& & \mathcal{G}_{1}(L):= \{ L'\in \mathcal{G}(L) \: |\: O(L')_{0}^{+} \ne O(L')_{0} \} , \\
& & \mathcal{G}_{2}(L):= \{ L'\in \mathcal{G}(L) \: |\: O(L')_{0}^{+} = O(L')_{0} \} .
\end{eqnarray*}  
We have the obvious decomposition 
$ \mathcal{G}(L)=\mathcal{G}_{1}(L)\sqcup \mathcal{G}_{2}(L)$. 
When $b_{+}$ is odd, 
we have 
$O(L)_{0}^{+} \ne O(L)_{0}$ 
if and only if 
$-{\rm id\/}_{D_{L}}$ is contained in $r_{L}(O(L)^{+})$.


\subsection{Twisted $K3$ surfaces}

Let $(S, \alpha )$ be a twisted $K3$ surface, 
where $\alpha $ is an element of the Brauer group ${\rm Br\/}(S)$.
Via the exponential sequence 
\begin{equation*}
0 \to Pic(S) \to H^{2}(S, {\Z}) \to H^{2}(\mathcal{O} _{S}) \to H^{2}(\mathcal{O} _{S}^{\times}) \to 1, 
\end{equation*} 
there is a canonical isomorphism 
${\rm Br\/}(S)\simeq {\rm Hom\/}(T(S), {\Q}/{\Z})$.
We identify 
a twisting $\alpha \in {\rm Br\/}(S)$ 
with 
a surjective homomorphism $\alpha : T(S)\to {\Z}/{\rm ord\/}(\alpha ){\Z}$, 
whose kernel is denoted by $T(S, \alpha )$. 
So 
the identity of ${\rm Br\/}(S)$ is denoted by $0$ 
and the inverse of $\alpha \in {\rm Br\/}(S)$ 
is denoted by $-\alpha $.

\begin{definition}
Set 
\begin{eqnarray*}
& & O_{Hodge}(T(S), \alpha ):= \{ g\in O_{Hodge}(T(S)), \: g^{\ast}\alpha =\alpha \} , \\
& & \Gamma (S, \alpha ):= r_{\widetilde{NS}(S)}^{-1} ( \lambda \circ r_{T(S)} (O_{Hodge}(T(S), \alpha ))), 
\end{eqnarray*} 
where 
$O_{Hodge}(T(S))$ is the group of the Hodge isometries of $T(S)$, 
and 
$\lambda : O(D_{T(S)})\stackrel{\simeq}{\to} O(D_{\widetilde{NS}(S)})$ 
is the isomorphism induced from the isometry 
$\lambda _{\widetilde{H}(S, {\Z})} : (D_{T(S)}, q) \simeq (D_{\widetilde{NS}(S)}, -q)$. 
\end{definition}

From the inclusions 
$T(S, \alpha ) \subset T(S) \subset T(S, \alpha )^{\vee}$, 
we have the isomorphism 
\[
O_{Hodge}(T(S), \alpha ) \simeq 
\Bigl\{ 
g \in O_{Hodge}(T(S, \alpha )), \: \:
r(g)(\alpha ^{-1}(\bar{1})) = \alpha ^{-1}(\bar{1}) \; \in D_{T(S, \alpha )} 
\Bigr\} , 
\]
where $\bar{1} \in {\Z}/{\rm ord\/}(\alpha ){\Z}$. 
In generic case, 
$O_{Hodge}(T(S), \alpha )= \{ {\rm id\/} \}$ 
if ${\rm ord\/}(\alpha ) \geq 3$
and 
$O_{Hodge}(T(S), \alpha )= \{ \pm {\rm id\/} \}$ 
if ${\rm ord\/}(\alpha ) \leq 2$. 

We need the following corollary of Huybrechts-Stellari's theorem.

\begin{prop}[\cite{Ca1}, \cite{H-S1}, \cite{Yo}, \cite{H-S2}]\label{H-S}
Let $S$ be a $K3$ surface 
and let $(S', \alpha ')$ be a twisted $K3$ surface. 
Then 
there exists an equivalence $D^{b}(S', \alpha ') \simeq D^{b}(S)$ 
if and only if 
there exists a Hodge isometry $T(S)\simeq T(S', \alpha ')$.
\end{prop}

\begin{proof}
By the theorem of Canonaco-Stellari (\cite{C-S}), 
every equivalence between derived categories of twisted $K3$ surfaces 
is of Fourier-Mukai type. 
Thus only "if" part requires an explanation.

Assume that 
there is a Hodge isometry $T(S)\simeq T(S', \alpha ')$. 
Let $B'\in H^{2}(S', {\Q})$ be a B-field lift of $\alpha '$.
Since there is a Hodge isometry 
$T(S', \alpha ')\simeq T(S', B')$, 
we obtain a Hodge isometry $g: T(S)\simeq T(S', B')$. 
By Proposition \ref{Nikulin1}, 
$\widetilde{NS}(S)$ 
and 
$\widetilde{NS}(S', B')$ 
are isogenus. 
Because 
$\widetilde{NS}(S)$ admits an embedding of the hyperbolic plane $U$, 
it follows from Proposition \ref{Nikulin2} that 
$g$ 
extends to a Hodge isometry 
$\widetilde{\Phi} : \widetilde{H}(S, {\Z}) \simeq \widetilde{H}(S', B', {\Z})$. 
Composing $\widetilde{\Phi}$   
with the isometry 
$-{\rm id\/}_{H^{0}(S)+H^{4}(S)}\oplus {\rm id\/}_{H^{2}(S)}$ 
if necessary, 
we may assume that 
$\widetilde{\Phi}$ is orientation-preserving. 
Now we have 
$D^{b}(S', \alpha ') \simeq D^{b}(S)$ 
by Huybrechts-Stellari's theorem. 
\end{proof}

When $\alpha '=0$, 
Proposition \ref{H-S} is 
Mukai-Orlov's theorem (\cite{Mu}, \cite{Or}).

\begin{definition}
For a twisted $K3$ surface $(S, \alpha )$, set
\[
{\rm FM\/}^{d}(S, \alpha ) := 
\Bigl\{  \: 
(S', \alpha ') \: : {\rm twisted\/} \: K3 \: {\rm surface\/} ,\: 
D^{b}(S', \alpha ')\simeq D^{b}(S, \alpha ),\: 
{\rm ord\/}(\alpha ')=d \: \Bigr\} / \simeq .
\]
A member $(S', \alpha ')\in {\rm FM\/}^{d}(S, \alpha )$
is called a 
{\it twisted Fourier-Mukai\/} (FM) partner of $(S, \alpha )$. 
It is easily seen that 
${\rm FM\/}^{d}(S, \alpha )$ is empty 
unless 
$d^{2}$ divides ${\rm det\/}(T(S, \alpha ))$.
\end{definition}

\begin{definition}
Set
\[
\mathcal{FM}^{d}(S, \alpha ) := {\rm FM\/}^{d}(S, \alpha ) / \sim ,
\]
where 
two twisted FM partners 
$(S_{1}, \alpha _{1})$, 
$(S_{2}, \alpha _{2}) \in {\rm FM\/}^{d}(S, \alpha )$ 
are equivalent 
if there exists a Hodge isometry 
$g : T(S_{1})\simeq T(S_{2}) $ 
with 
$g^{\ast }\alpha _{2}=\alpha _{1}$. 
Denote by 
\[
\pi : {\rm FM\/}^{d}(S, \alpha ) \twoheadrightarrow \mathcal{FM}^{d}(S, \alpha ) 
\]
the quotient map.
\end{definition}

We remark that 
for a twisted $K3$ surface $(S, \alpha )$ 
with a B-field lift $B\in H^{2}(S, {\Q})$, 
the period of $S$ is determined by the natural Hodge isometry 
\[
H^{2}(S, {\Z})\simeq 
\Bigl( (0, 0, 1)^{\perp}\cap \widetilde{H}(S, B, {\Z})\Bigr) / {\Z}(0, 0, 1) .
\]

In Section 5, 
we shall calculate twistings in terms of discriminant group.

\begin{lemma}\label{twist and disc form}
Let $(S, \alpha )$ be a twisted $K3$ surface with a B-field lift $B\in H^{2}(S, {\Q})$. 
We denote 
$d:={\rm ord\/}(\alpha )$
and 
$\lambda := \lambda_{ \widetilde{H}(S, B, {\Z})}: 
(D_{\widetilde{NS}(S, B)}, q) \simeq (D_{T(S, B)}, -q)$. 
Then 
$(0, 0, \frac{-1}{d}) \in \widetilde{NS}(S, B)^{\vee}$ 
and 
we have 
the Hodge isometry 
\begin{equation}\label{eqn: exp(B) and disc form}
{\rm e\/}^{B} := {\rm exp\/}(B)\wedge : 
T(S) 
\stackrel{\simeq}{\to} 
\Bigl\langle 
\lambda ((0, 0, \frac{-1}{d})), \: T(S, B) 
\Bigr\rangle . 
\end{equation} 
The twisting 
$\alpha $ 
is given by 
\begin{equation}\label{eqn: twisting and disc form}
T(S) 
\stackrel{{\rm e\/}^{B}}{\to} 
{\rm e\/}^{B}(T(S))/T(S, B) 
\simeq 
\Bigl\langle 
\lambda ((0, 0, \frac{-1}{d})) 
\Bigr\rangle 
\simeq 
{\Z}/d{\Z} .
\end{equation}
\end{lemma}

\begin{proof}
Since  
$\alpha : T(S)\to {\Z}/d{\Z}$ 
is surjective, 
we can take a transcendental cycle 
$l\in T(S)$ 
such that 
$(l, B) = \frac{1}{d} + k$ with $k\in {\Z}$. 
Setting 
$l' := {\rm e\/}^{B}(l)=l+(0, 0, \frac{1}{d}+k) \in T(S, B)^{\vee}$, 
we have 
${\rm e\/}^{B}:T(S)\simeq \langle l', \: T(S, B)\rangle $. 
Since  
$l'-(0, 0, \frac{1}{d}) \in \widetilde{H}(S, B, {\Z})$, 
it follows that 
$((0, 0, \frac{-1}{d}), m)\in {\Z}$ 
for all $m\in \widetilde{NS}(S, B)$. 
By the relation 
$\lambda ((0, 0, \frac{-1}{d}))=l' \in D_{T(S, B)}$, 
we have (\ref{eqn: exp(B) and disc form}). 
Then (\ref{eqn: twisting and disc form}) 
follows from the relation 
$\alpha (l)=\bar{1} \in {\Z}/d{\Z}$. 
\end{proof}

More precisely, 
the divisibility of 
$(0, 0, 1) \in \widetilde{NS}(S, B)$ 
is equal to $d$, 
but 
we do not need this fact in the following.

\section{Lattice-theoretic descriptions}

\subsection{Isotropic elements of the discriminant form}

Let 
$(S, \alpha )$ 
be a twisted $K3$ surface. 
We write $T:=T(S, \alpha )$, 
which is an even lattice of ${\rm sign\/}(T)=(2, 20-\rho (S))$ 
equipped with a period. 
For a twisted FM partner 
$(S_{1}, \alpha _{1})\in {\rm FM\/}^{d}(S, \alpha )$ 
there exists a Hodge isometry 
$g_{1}:T(S_{1}, \alpha _{1})\simeq T$ 
by Huybrechts-Stellari's theorem (\cite{H-S1}).  
Then 
$g_{1}$ and $\alpha _{1}$ 
induce an isomorphism 
$g_{1}(T(S_{1}))/T \stackrel{\simeq}{\to} {\Z}/d{\Z}$. 
The subgroup
$g_{1}(T(S_{1}))/T \subset D_{T}$ 
is an isotropic cyclic group of order$=d$.

\begin{definition}
Define the map 
\[
\mu : \mathcal{FM}^{d}(S, \alpha ) \to O_{Hodge}(T)\backslash I^{d}(D_{T})
\]
by 
$\mu ([(S_{1}, \alpha _{1})]):=[g_{1}(\alpha _{1}^{-1}(\bar{1}))]$. 
\end{definition}

\begin{lemma}\label{injective 1}
The map $\mu $ is well-defined and is injective.
\end{lemma}

\begin{proof}
The ambiguity of the choice of a Hodge isometry 
$g_{1}: T(S_{1}, \alpha _{1})\simeq T$ 
comes from the action of $O_{Hodge}(T)$ 
so that 
the class  
$[g_{1}(\alpha _{1}^{-1}(\bar{1}))] \in O_{Hodge}(T)\backslash I^{d}(D_{T})$ 
is well-defined from the partner $(S_{1}, \alpha _{1})$.
Since 
the map $\mu $ is defined 
in terms of $T(S_{1})$ and $\alpha _{1}$, 
it is clear that 
$\mu $ is well-defined as a map from $\mathcal{FM}^{d}(S, \alpha )$.

We prove the injectivity of $\mu $. 
For two partners 
$(S_{i}, \alpha _{i}) \in {\rm FM\/}^{d}(S, \alpha )$, 
$i=1, 2$,  
and Hodge isometries 
$g_{i}:T(S_{i}, \alpha _{i})\simeq T$,
write 
$x_{i}:=g_{i}(\alpha _{i}^{-1}(\bar{1})) \in I^{d}(D_{T})$. 
Assume  
the existence of a Hodge isometry 
$\varphi \in O_{Hodge}(T)$ such that 
$x_{2}=r(\varphi )(x_{1})$. 
Since 
$r(\varphi )(\langle x_{1} \rangle )=\langle x_{2} \rangle \subset D_{T}$, 
$\varphi $ extends to the Hodge isometry 
$\widetilde{\varphi} : T(S_{1})\simeq T(S_{2})$. 
By the equality 
$r(\varphi ) ( g_{1}(\alpha _{1}^{-1}(\bar{1})) ) = g_{2}(\alpha _{2}^{-1}(\bar{1}))$ 
we have 
$\widetilde{\varphi} ^{\ast}\alpha _{2}=\alpha _{1}$. 
Hence 
$[(S_{1}, \alpha _{1})]=[(S_{2}, \alpha _{2})]\in \mathcal{FM}^{d}(S, \alpha )$.
\end{proof}

\begin{prop}\label{bijective if untwisted 1}
If 
the twisted $K3$ surface $(S, \alpha )$ 
has an untwisted FM partner,  
then the map $\mu $ is bijective.
\end{prop}

\begin{proof}
It suffices to show the surjectivity. 
Let 
$S'$ be an untwisted FM partner of $(S, \alpha )$. 
Since there exists a Hodge isometry 
$T(S')\simeq T$, 
we may assume from the first that 
$(S, \alpha )$ itself is untwisted. 
Take an isotropic element 
$x\in I^{d}(D_{T}) = I^{d}(D_{T(S)})$. 
Via the isometry 
$\lambda : (D_{T(S)}, q)\simeq (D_{\widetilde{NS}(S)}, -q)$, 
we obtain an isotropic element 
$\lambda (x)\in I^{d}(D_{\widetilde{NS}(S)})$. 
Set 
\begin{eqnarray*}
& &\widetilde{M}_{x}:=\langle \lambda (x), \: \widetilde{NS}(S) \rangle \: \subset \widetilde{NS}(S)^{\vee} , \\
& &T_{x}:=\langle x, \: T(S) \rangle \: \subset T(S)^{\vee} ,
\end{eqnarray*}
which are even overlattices of 
$\widetilde{NS}(S)$, $T(S)$ respectively.

Via the isometry 
\[
(D_{T_{x}}, q) \simeq 
(\langle x \rangle ^{\perp} /\langle x \rangle , q) \simeq 
(\langle \lambda (x) \rangle ^{\perp} /\langle \lambda (x) \rangle , -q) \simeq 
(D_{\widetilde{M}_{x}}, -q),  
\]
we obtain an embedding 
$\widetilde{M}_{x}\oplus T_{x}\hookrightarrow \widetilde{\Lambda}_{K3}$ 
with 
both $\widetilde{M}_{x}$ and $T_{x}$ embedded primitively. 
Since 
$\widetilde{NS}(S) \subset \widetilde{M}_{x}$, 
the lattice 
$\widetilde{M}_{x}$ 
admits an embedding 
$\varphi $ of the hyperbolic plane $U$. 
Then the lattice 
$\Lambda _{\varphi}:=\varphi (U)^{\perp}\cap \widetilde{\Lambda}_{K3}$ 
is isometric to the $K3$ lattice $\Lambda _{K3}$ 
and has the period induced from $T_{x}$. 
Now 
the surjectivity of the period map (\cite{To}, \cite{K3}) 
assures the existences of 
a $K3$ surface $S_{\varphi}$ 
and a Hodge isometry 
$\Phi : H^{2}(S_{\varphi}, {\Z}) \stackrel{\simeq}{\to} \Lambda _{\varphi}$. 
Pulling back the homomorphism 
\[
\alpha _{x}: 
T_{x} \to T_{x}/T(S) 
\simeq \langle x\rangle \simeq {\Z}/d{\Z}, \: \: \: 
\alpha _{x}(x)=\bar{1}
\]
by 
$\Phi |_{T(S_{\varphi})}$, 
we obtain a twisted $K3$ surface 
$(S_{\varphi}, \alpha _{\varphi})$.

Since there exists a Hodge isometry 
$T(S_{\varphi}, \alpha _{\varphi}) \simeq {\rm Ker\/}(\alpha _{x}) = T(S)$, 
it follows from Proposition \ref{H-S} that   
$(S_{\varphi}, \alpha _{\varphi}) 
\in {\rm FM\/}^{d}(S)$.  
By the construction, 
we have 
$\mu ([(S_{\varphi}, \alpha _{\varphi})])=[x]$.
\end{proof}

In general, 
there is a condition on the image of $\mu $. 
Let 
$B\in H^{2}(S, {\Q})$ be a B-field lift of $\alpha \in {\rm Br\/}(S)$. 
There is a natural isometry 
$\lambda : (D_{T}, q) \simeq (D_{\widetilde{NS}(S, B)}, -q)$.
We define 
\begin{equation}\label{eqn:J^{d}(D)}
J^{d}(D_{T}):= 
\{ x\in I^{d}(D_{T}), \: \: 
{\rm Emb\/}(U, \langle \lambda (x), \widetilde{NS}(S, B) \rangle ) 
\ne \phi \} .
\end{equation}
One can verify that 
the set $J^{d}(D_{T})$ is 
independent of the choice of a lift $B$.

\begin{prop}
We have the inclusion 
\[
{\rm Im\/}(\mu )\subset O_{Hodge}(T)\backslash J^{d}(D_{T}).
\]
\end{prop}

\begin{proof}
For a twisted FM partner $(S_{1}, \alpha _{1})\in {\rm FM\/}^{d}(S, \alpha )$, 
write $[x]:=\mu ([(S_{1}, \alpha _{1})])$. 
By Proposition \ref{Nikulin2}, 
it suffices to show that 
$\widetilde{NS}(S_{1})$ and 
$\widetilde{M}:= \langle \lambda (x), \widetilde{NS}(S, B) \rangle $
are isogenus. 
We have the isometry 
\begin{eqnarray*}
(D_{\widetilde{NS}(S_{1})}, q) 
& \simeq & 
(D_{T(S_{1})}, -q) \simeq (\langle x \rangle ^{\perp} /\langle x \rangle , -q) \\ 
& \simeq & 
(\langle \lambda (x) \rangle ^{\perp} /\langle \lambda(x) \rangle , q) \simeq (D_{\widetilde{M}}, q).
\end{eqnarray*}
\end{proof}

We make the following identification 
tacitly in Section 4. 
For 
$x\in I^{d}(D_{T})$ 
such that 
$[x]=\mu ([(S_{1}, \alpha _{1})])$, 
set 
$T_{x}:= \langle x, T \rangle $.  
For a Hodge isometry 
$g_{1}:T(S_{1}, \alpha _{1}) \simeq T$ 
with 
$g_{1}(\alpha _{1}^{-1}(\bar{1}))=x$, 
by Proposition \ref{Nikulin2} 
there exists an isometry 
\[
\gamma : 
\widetilde{NS}(S_{1}) 
\stackrel{\simeq}{\to}  
\langle \lambda (x), \widetilde{NS}(S, B) \rangle \: \: \: \: \: 
{\rm with\/} \: \: \: 
r(\gamma ) \circ \lambda _{\widetilde{H}(S_{1}, {\Z})} = \bar{\lambda } \circ r(\bar{g_{1}}) . 
\]
Here 
the Hodge isometry 
$\bar{g_{1}} : T(S_{1}) \simeq T_{x}$ 
is induced from $g_{1}$, 
and 
the isometry 
$\bar{\lambda } : 
(D_{T_{x}}, q) 
\simeq 
(D_{\langle \lambda (x), \widetilde{NS}(S, B) \rangle }, -q)$ 
is induced from 
$\lambda $.

\subsection{Embeddings of the hyperbolic plane}

Let 
$(S_{1}, \alpha _{1}) \in {\rm FM\/}^{d}(S, \alpha )$
be a twisted FM partner of $(S, \alpha )$. 
Recall from Section 2.2 that 
we have defined the finite-index subgroup 
$\Gamma (S_{1}, \alpha _{1}) \subset O(\widetilde{NS}(S_{1}))$. 
The subgroup 
$\Gamma (S_{1}, \alpha _{1})^{+} \subset \Gamma (S_{1}, \alpha _{1})$ 
consists of 
orientation-preserving isometries in 
$\Gamma (S_{1}, \alpha _{1})$ (see Section 2.1). 
We define the map 
\[
\nu : \pi ^{-1} \Bigl( \pi ((S_{1}, \alpha _{1}))\Bigr) \to 
\Gamma (S_{1}, \alpha _{1})^{+} \backslash {\rm Emb\/} (U, \widetilde{NS}(S_{1}))
\]
as follows. 
For a twisted FM partner  
$(S_{2}, \alpha _{2}) \in \pi ^{-1} \Bigl( \pi ((S_{1}, \alpha _{1}))\Bigr) $ 
there exists a Hodge isometry 
$g:T(S_{2})\stackrel{\simeq}{\to} T(S_{1})$ 
with 
$g^{\ast}\alpha _{1}=\alpha _{2}$. 
By Proposition \ref{Nikulin2}, 
$g$ can be extended to a Hodge isometry 
$\widetilde{\Phi} : \widetilde{H}(S_{2}, {\Z}) \stackrel{\simeq}{\to} \widetilde{H}(S_{1}, {\Z})$ 
such that 
the isometry 
$\widetilde{\Phi} |_{\widetilde{NS}(S_{2})} : 
\widetilde{NS}(S_{2}) \stackrel{\simeq}{\to} \widetilde{NS}(S_{1}) $ 
is orientation-preserving. 
Then, 
by considering 
$\widetilde{\Phi} : H^{0}(S_{2}, {\Z})+H^{4}(S_{2}, {\Z}) \hookrightarrow \widetilde{NS}(S_{1}) $, 
we obtain an embedding 
$\varphi : U \hookrightarrow \widetilde{NS}(S_{1})$. 
Here 
we identify 
$H^{0}(S_{2}, {\Z})+H^{4}(S_{2}, {\Z})$ 
with 
$U$ by 
identifying $(1, 0, 0)$ with $e$, 
and $(0, 0, -1)$ with $f$.

\begin{lemma}\label{injective 2}
The assignment 
$(S_{2}, \alpha _{2}) \mapsto [\varphi ]$ 
defines an injective map 
\[
\nu : 
\pi ^{-1} \Bigl( \pi ((S_{1}, \alpha _{1}))\Bigr) 
\to 
\Gamma (S_{1}, \alpha _{1})^{+} \backslash {\rm Emb\/} (U, \widetilde{NS}(S_{1})).
\]
\end{lemma}

\begin{proof}
Firstly we prove that $\nu $ is well-defined.
Let us given two Hodge isometries 
$g, g' : T(S_{2})\stackrel{\simeq}{\to} T(S_{1})$ 
with 
$g^{\ast}\alpha _{1} = (g')^{\ast}\alpha _{1} = \alpha _{2}$. 
Let 
$\gamma , \gamma ' : \widetilde{NS}(S_{2}) \stackrel{\simeq}{\to} \widetilde{NS}(S_{1})$ 
be orientation-preserving isometries 
such that 
$\gamma \oplus g$ and $\gamma '\oplus g'$ 
extend to Hodge isometries 
$\widetilde{H}(S_{2}, {\Z}) \simeq \widetilde{H}(S_{1}, {\Z})$. 
Then 
the isometry 
$\gamma '\circ \gamma ^{-1} \in O(\widetilde{NS}(S_{1}))$ 
preserves the orientation 
and satisfies 
$\lambda \circ r(\gamma '\circ \gamma ^{-1}) = r(g'\circ g^{-1}) \circ \lambda $, 
so that we have 
$\gamma '\in \Gamma (S_{1}, \alpha _{1})^{+}\cdot \gamma $.  
Hence 
the map $\nu $ is well-defined.

We prove the injectivity of $\nu $. 
For two partners  
$(S_{i}, \alpha _{i}) \in \pi ^{-1} \Bigl( \pi ((S_{1}, \alpha _{1}))\Bigr) $, 
$i=2, 3$, 
set 
$[\varphi _{i}]:= \nu ((S_{i}, \alpha _{i}))$ 
and assume the existence of an isometry 
$\gamma \in \Gamma (S_{1}, \alpha _{1})^{+}$ 
such that 
$\varphi _{2}=\gamma \circ \varphi _{3}$. 
We can extend $\gamma $ to a Hodge isometry 
$\widetilde{\Phi} : \widetilde{H}(S_{1}, {\Z}) \stackrel{\simeq}{\to}  \widetilde{H}(S_{1}, {\Z})$ 
with 
$(\widetilde{\Phi} |_{T(S_{1})})^{\ast}\alpha _{1} = \alpha _{1}$. 
By restricting $\widetilde{\Phi}$ to 
$\varphi _{2}(U)^{\perp} \cap \widetilde{H}(S_{1}, {\Z})$, 
we obtain a Hodge isometry 
$\Phi : H^{2}(S_{2}, {\Z}) \stackrel{\simeq}{\to} H^{2}(S_{3}, {\Z})$ 
with 
$(\Phi |_{T(S_{2})})^{\ast}\alpha _{3} = \alpha _{2}$. 
Since 
$\gamma $ preserves the orientation, 
$\Phi $ maps 
the positive cone $NS(S_{2})^{+}$ to 
the positive cone $NS(S_{3})^{+}$.
Composing $\Phi $ with an element of the Weyl group $W(S_{3})$, 
we may assume that $\Phi $ is effective. 
Then 
it follows from the Torelli theorem (\cite{PS-S}, \cite{K3}) that 
$(S_{3}, \alpha _{3}) \simeq (S_{2}, \alpha _{2})$. 
\end{proof}

\begin{prop}\label{bijective if untwisted 2}
If $(S, \alpha )$ has an untwisted FM partner, 
then the map $\nu $ 
for every twisted FM partner 
$(S_{1}, \alpha _{1}) \in {\rm FM\/}^{d}(S, \alpha )$ 
is bijective.
\end{prop}

\begin{proof} 
We may assume that 
$(S, \alpha )$ itself is untwisted.
To prove the surjectivity, 
we shall construct a twisted $K3$ surface 
$(S_{\varphi}, \alpha _{\varphi})$ 
from an arbitrary embedding 
$\varphi : U \hookrightarrow \widetilde{NS}(S_{1})$. 
The lattice 
$\Lambda _{\varphi} := \varphi (U)^{\perp}\cap \widetilde{H}(S_{1}, {\Z})$ 
is isometric to the $K3$ lattice $\Lambda _{K3}$ 
and possesses the period induced from $T(S_{1})$. 
On the other hand, 
the lattice 
$M_{\varphi}:= \varphi (U)^{\perp}\cap \widetilde{NS}(S_{1})$ 
has the orientation 
induced from $\varphi $ and the orientation of $\widetilde{NS}(S_{1})$. 
That is, 
we can choose a connected component $M_{\varphi}^{+}$ 
of the open set 
$\{ v\in M_{\varphi}\otimes {\R}, \: (v,v)>0 \}$ 
so that 
for each vector 
$v \in  M_{\varphi}^{+}$ 
the oriented positive-definite two-plane 
${\R}\varphi (e+f) \oplus {\R}v \subset \widetilde{NS}(S_{1}) \otimes {\R}$ 
is of positive orientation.  
By the surjectivity of the period map, 
there exist a $K3$ surface $S_{\varphi}$ 
and a Hodge isometry 
$\Phi : H^{2}(S_{\varphi}, {\Z}) \stackrel{\simeq}{\to} \Lambda _{\varphi}$ 
such that 
$\Phi (NS(S_{\varphi})^{+})= M_{\varphi}^{+}$. 
Pulling back $\alpha _{1}$ by $\Phi |_{T(S_{\varphi})}$, 
we obtain a twisted $K3$ surface 
$(S_{\varphi}, \alpha _{\varphi})$.

Since there is a Hodge isometry 
\[
T(S_{\varphi}, \alpha _{\varphi}) \simeq  T(S_{1}, \alpha _{1}) \simeq T(S),  
\] 
we have an equivalence 
$D^{b}(S_{\varphi}, \alpha _{\varphi}) \simeq D^{b}(S)$ 
by Proposition \ref{H-S}. 
Hence 
we have 
$(S_{\varphi}, \alpha _{\varphi}) \in \pi ^{-1} \Bigl( \pi ((S_{1}, \alpha _{1}))\Bigr) $.

Since 
the direct sum of 
the Hodge isometry 
$\Phi |_{T(S_{\varphi})} : T(S_{\varphi}) \stackrel{\simeq}{\to} T(S_{1})$ 
satisfying 
$(\Phi |_{T(S_{\varphi})})^{\ast}\alpha _{1} = \alpha _{\varphi}$ 
and 
the orientation-preserving isometry 
\[
\varphi \oplus ( \Phi |_{NS(S_{\varphi})} ) : 
\widetilde{NS}(S_{\varphi}) 
\stackrel{\simeq}{\to} 
\varphi (U)\oplus M_{\varphi} = \widetilde{NS}(S_{1})
\] 
can be extended to the Hodge isometry 
$\varphi \oplus \Phi $, 
we have 
$\nu ((S_{\varphi}, \alpha _{\varphi}))=[\varphi ]$. 
\end{proof}

\begin{remark}\label{sharpness and a conjecture}
In his thesis \cite{Ca1}, 
C\u ald\u araru proposed the following problem: 

\begin{question}[\cite{Ca1}]
Let $(S_{1}, \alpha _{1})$ be a twisted $K3$ surface. 
For each untwisted FM partner 
$S_{2} \in {\rm FM\/}(S_{1})$ 
and each Hodge isometry 
$g: T(S_{2}) \stackrel{\simeq}{\to} T(S_{1})$, 
we obtain a twisted $K3$ surface 
$(S_{2}, g^{\ast}\alpha _{1})$. 
Is it true that 
$(S_{2}, g^{\ast}\alpha _{1}) \in {\rm FM\/}^{d}(S_{1}, \alpha _{1})$ ?
\end{question}

From the construction of the twisted $K3$ surface 
$(S_{\varphi}, \alpha _{\varphi})$ 
in the proof of Proposition \ref{bijective if untwisted 2}, 
it is immediately verified that 
the map $\nu $ in Lemma \ref{injective 2} is bijective 
if and only if 
the answer to C\u ald\u araru's question is positive  
for the twisted $K3$ surface $(S_{1}, \alpha _{1})$ 
and each $S_{2}$, $g$.

\end{remark}

\subsection{Union of the two orbits}

Given an embedding 
$\varphi : U \hookrightarrow \widetilde{NS}(S_{1})$, 
we have the decomposition 
\[
\Gamma (S_{1}, \alpha _{1}) = 
\Gamma (S_{1}, \alpha _{1})^{+} \sqcup \:
\Gamma (S_{1}, \alpha _{1})^{+}\cdot ( -{\rm id\/}_{\varphi (U)} \oplus {\rm id\/}_{\varphi (U)^{\perp}} ).  
\] 
From this we obtain 
\[
\Gamma (S_{1}, \alpha _{1}) \cdot \varphi 
\; = \; 
\Gamma (S_{1}, \alpha _{1})^{+}\cdot \varphi 
\: \cup \: 
\Gamma (S_{1}, \alpha _{1})^{+}\cdot (-\varphi ).
\]

\begin{lemma}\label{same orbit}
Let 
$(S_{\varphi}, \alpha _{\varphi})$ 
be the twisted $K3$ surface 
constructed in the proof of Proposition $\ref{bijective if untwisted 2}$. 
Then 
$-\varphi \in \Gamma (S_{1}, \alpha _{1}) ^{+} \cdot \varphi $ 
if and only if 
$(S_{\varphi}, \alpha _{\varphi}) \simeq (S_{\varphi}, -\alpha _{\varphi})$.   
\end{lemma}

\begin{proof} 
It suffices to show that 
$(S_{-\varphi}, \alpha _{-\varphi}) \simeq (S_{\varphi}, -\alpha _{\varphi})$.   
Since 
$M_{-\varphi}^{+} = -M_{\varphi}^{+}$, 
we have a Hodge isometry 
$\Phi : H^{2}(S_{-\varphi}, {\Z}) \stackrel{\simeq}{\to} H^{2}(S_{\varphi}, {\Z})$ 
with 
$(\Phi |_{T(S_{-\varphi})})^{\ast}\alpha _{\varphi} = \alpha _{-\varphi}$ 
and 
$\Phi (NS(S_{-\varphi})^{+}) = -NS(S_{\varphi})^{+}$. 
Then 
$-\Phi : H^{2}(S_{-\varphi}, {\Z}) \stackrel{\simeq}{\to} H^{2}(S_{\varphi}, {\Z})$ 
satisfies 
$(-\Phi |_{T(S_{-\varphi})})^{\ast}(-\alpha _{\varphi}) = \alpha _{-\varphi}$ 
and 
$-\Phi (NS(S_{-\varphi})^{+}) = NS(S_{\varphi})^{+}$. 
\end{proof}

When 
${\rm ord\/}(\alpha _{1}) \leq 2$, 
we have 
$\alpha _{\varphi} = -\alpha _{\varphi}$
so that 
$\Gamma (S_{1}, \alpha _{1}) \cdot \varphi = \Gamma (S_{1}, \alpha _{1}) ^{+} \cdot \varphi $ 
for every embedding $\varphi $. 
Hence we obtain the identification 
\[
\Gamma (S_{1}, \alpha _{1})^{+} \backslash {\rm Emb\/} (U, \widetilde{NS}(S_{1})) 
\: \simeq \: 
\Gamma (S_{1}, \alpha _{1})     \backslash {\rm Emb\/} (U, \widetilde{NS}(S_{1})). 
\]

\begin{prop}\label{anti-symplectic}
Assume that 
${\rm ord\/}(\alpha _{1}) \geq 3$. 
For an embedding  
$\varphi : U \hookrightarrow \widetilde{NS}(S_{1})$ 
the following conditions are equivalent. 

$({\rm i\/})$  
$(S_{\varphi}, \alpha _{\varphi}) \simeq (S_{\varphi}, -\alpha _{\varphi})$.   
 
$({\rm ii\/})$  
$S_{\varphi}$ admits an anti-symplectic automorphism. 

$({\rm iii\/})$  
$\varphi (U)^{\perp} \cap \widetilde{NS}(S_{1}) \in \mathcal{G}_{1}(NS(S_{1}))$, 
i.e, 
$O(NS(S_{\varphi}))_{0}^{+} \ne O(NS(S_{\varphi}))_{0}$. 
\end{prop}

\begin{proof}
(ii) $\Rightarrow $ (i) is obvious.

(i) $\Rightarrow $ (ii) : 
Assume the existence of an automorphism 
$f : (S_{\varphi}, \alpha _{\varphi}) \simeq (S_{\varphi}, -\alpha _{\varphi})$. 
Since 
${\rm ord\/}(\alpha _{\varphi}) \geq 3$, 
${\rm ord\/}(f^{\ast}|_{T(S_{\varphi})})$ must be even, 
say $2n$. 
As $O_{Hodge}(T(S_{\varphi}))$ is a cyclic group 
(cf. \cite{H-L-O-Y}), 
then $f^{n}$ is an anti-symplectic automorphism of $S_{\varphi}$.

(ii) $\Rightarrow $ (iii) :
For an anti-symplectic automorphism $f$, 
we have 
$-f^{\ast}|_{NS(S_{\varphi})} \in O(NS(S_{\varphi}))_{0} - O(NS(S_{\varphi}))_{0}^{+}$.

(iii) $\Rightarrow $ (ii) : 
Assume the existence of an isometry 
$\gamma \in O(NS(S_{\varphi}))_{0} - O(NS(S_{\varphi}))_{0}^{+}$. 
Then we can extend 
the isometry 
$-\gamma \oplus -{\rm id\/}_{T(S_{\varphi})}$ 
to the anti-symplectic Hodge isometry 
$\Phi : H^{2}(S_{\varphi}, {\Z}) \simeq H^{2}(S_{\varphi}, {\Z})$ 
which preserves the positive cone. 
Composing $\Phi $ with an element of the Weyl group $W(S_{\varphi})$, 
we may assume that $\Phi $ is effective. 
\end{proof}

\section{The counting formula}

\begin{prop}\label{number of orbits}
For a twisted $K3$ surface $(S_{1}, \alpha _{1})$ 
there exists a bijection 
\[
\Gamma (S_{1}, \alpha _{1}) \backslash {\rm Emb\/} (U, \widetilde{NS}(S_{1})) 
\simeq 
\mathop{\bigsqcup}_{M\in \mathcal{G}(NS(S_{1}))} 
\Bigl( 
O_{Hodge}(T(S_{1}), \alpha _{1})  \backslash  O(D_{M})  /  O(M)
\Bigr).
\]
\end{prop}

\begin{proof}
We can decompose the set 
${\rm Emb\/}(U, \widetilde{NS}(S_{1}))$ as 
\[
{\rm Emb\/}(U, \widetilde{NS}(S_{1})) = 
\mathop{\bigsqcup}_{M\in \mathcal{G}(NS(S_{1}))} 
\left\{ 
\: \varphi \in {\rm Emb\/}(U, \widetilde{NS}(S_{1})), \: \: \varphi (U)^{\perp} \simeq M \: 
\right\} .
\]
The isometry group 
$O(\widetilde{NS}(S_{1}))=O(M \oplus U)$ 
acts 
transitively on each component 
$
\{ \, \varphi \in {\rm Emb\/}(U, \widetilde{NS}(S_{1})), \: \: \varphi (U)^{\perp} \simeq M \: \} 
$
with the stabilizer subgroup 
$O(M)\subset O(M\oplus U)$. 
From the surjectivity of  
$r : O(M \oplus U) \to O(D_{M \oplus U}) \simeq O(D_{M})$ 
(Proposition 2.2) 
and the inclusion 
$O(M \oplus U)_{0} \subset \Gamma (S_{1}, \alpha _{1})$, 
we have 
\begin{eqnarray*}
 \Gamma (S_{1}, \alpha _{1}) \backslash {\rm Emb\/}(U, \widetilde{NS}(S_{1})) 
& \simeq & 
\mathop{\bigsqcup}_{M\in \mathcal{G}(NS(S_{1}))} 
\Bigl( 
\Gamma (S_{1}, \alpha _{1})  \backslash  O(M \oplus U)  /  O(M)
\Bigr) \\
& \simeq &
\mathop{\bigsqcup}_{M\in \mathcal{G}(NS(S_{1}))} 
\Bigl( 
r(\Gamma (S_{1}, \alpha _{1}) )  \backslash  O(D_{M})  /  r(O(M))
\Bigr) .
\end{eqnarray*}
\end{proof}

Let 
$(S, \alpha )$ 
be a twisted $K3$ surface 
with a B-field lift $B \in H^{2}(S, {\Q})$. 
For an isotropic element 
$x \in I^{d}(D_{T(S, \alpha )})$ 
we define the lattice $T_{x}$ 
by $T_{x}=\langle x, T(S, \alpha ) \rangle $ 
and define the homomorphism 
$\alpha _{x}:T_{x} \twoheadrightarrow {\Z}/d{\Z}$ 
by 
\[
\alpha _{x}:T_{x} \twoheadrightarrow T_{x}/T(S, \alpha ) \simeq \langle x \rangle \simeq {\Z}/d{\Z}, \: \: \: 
\alpha _{x}(x) = \bar{1} . 
\]
For a pair $(x, M)$ 
such that 
$x \in I^{d}(D_{T(S, \alpha )}) \stackrel{\lambda}{\simeq} I^{d}(D_{\widetilde{NS}(S, B )})$
and 
$\langle \lambda (x), \widetilde{NS}(S, B )\rangle \simeq U\oplus M$, 
we define the natural number 
$\tau (x, M)$ by 
\begin{equation*}
\tau (x, M) := \# 
\Bigl( 
O_{Hodge}(T_{x}, \alpha _{x}) \backslash O(D_{M}) /O(M) 
\Bigr) , 
\end{equation*}
where the group $O_{Hodge}(T_{x}, \alpha _{x})$ 
acts on $O(D_{M})$ via 
$r:O_{Hodge}(T_{x}, \alpha _{x}) \to O(D_{T_{x}})$ 
and 
$\bar{\lambda} : O(D_{T_{x}})\simeq O(D_{M\oplus U}) \simeq O(D_{M})$. 
For a natural number $d$ 
we define the natural number 
$\varepsilon (d)$ by 
\[
\varepsilon (d) = \left\{ 
\begin{array}{rl} 
1, & \: \: d=1, 2 \\
2, & \: \: d\geq 3.
\end{array} 
\right. 
\]

From 
Section 3 and Proposition \ref{number of orbits}, 
we deduce the formula for 
$\# {\rm FM\/}^{d}(S)$.

\begin{theorem}\label{the formula}
For a $K3$ surface $S$ 
the following formula holds. 
\begin{equation*}
\# {\rm FM\/}^{d}(S) = 
\mathop{\sum}_{x} \Bigl\{ 
\mathop{\sum}_{M} \tau (x, M) + \varepsilon (d) \mathop{\sum}_{M'} \tau (x, M')  \Bigr\} . 
\end{equation*}
Here 
$x$ runs over the set 
$O_{Hodge}(T(S)) \backslash I^{d}(D_{T(S)})$, 
and the lattices $M$, $M'$ 
run over the sets 
$\mathcal{G}_{1}(\langle \lambda (x), NS(S) \rangle )$, 
$\mathcal{G}_{2}(\langle \lambda (x), NS(S) \rangle )$ 
respectively.
\end{theorem}

We also obtain the following inequality.

\begin{prop}\label{the upper bound} 
For a twisted $K3$ surface $(S, \alpha )$ the following inequality holds. 
\begin{equation*}
\# {\rm FM\/}^{d}(S, \alpha ) \leq  
\mathop{\sum}_{x} \Bigl\{ 
\mathop{\sum}_{M} \tau (x, M) + \varepsilon (d) \mathop{\sum}_{M'} \tau (x, M')  \Bigr\} . 
\end{equation*}
Here 
$x$ 
runs over the set
$O_{Hodge}(T(S, \alpha )) \backslash J^{d}(D_{T(S, \alpha )})$, 
where 
$J^{d}(D_{T(S, \alpha )})$ is the set defined in $(\ref{eqn:J^{d}(D)})$. 
The lattices 
$M$, $M'$ 
run over the sets 
$\mathcal{G}_{1}(M_{\varphi})$, 
$\mathcal{G}_{2}(M_{\varphi})$ 
respectively, 
where $M_{\varphi}$ is a lattice satisfying 
$\langle \lambda (x), \widetilde{NS}(S, B) \rangle \simeq U\oplus M_{\varphi}$. 
\end{prop}

\begin{corollary}\label{Jacobian}
When $\widetilde{NS}(S, B)$ contains $U \oplus U$, 
then 
\[
\# {\rm FM\/}^{d}(S, \alpha ) = 
\# \left( O_{Hodge}(T(S, \alpha )) \backslash I^{d}(D_{T(S, \alpha )}) \right) .
\]
\end{corollary}

\begin{proof}
By considering 
$U^{\perp} \cap \widetilde{H}(S, B, {\Z})$, 
one observes the existence of 
an untwisted FM partner 
$S' \in {\rm FM\/}^{1}(S, \alpha )$. 
We have a Hodge isometry 
$T(S, \alpha ) \simeq T(S')$. 
Since 
$\widetilde{NS}(S')$ contains $U \oplus U$, 
then 
$NS(S')$ contains $U$. 
Now we apply Theorem \ref{the formula} to $S'$. 
For 
$x \in I^{d}(D_{T(S')})$, 
the lattice 
$M_{x}:=\langle \lambda (x), NS(S')\rangle $ 
admits an embedding of $U$ 
so that 
$\mathcal{G}(M_{x})=\{ M_{x} \} $ and 
$\tau (x, M_{x})=1$ 
by Proposition \ref{Nikulin2}. 
Then  
$\mathcal{G}(M_{x})=\mathcal{G}_{1}(M_{x})$ 
due to the existence of the isometry 
$-{\rm id\/}_{U}\oplus {\rm id\/}_{U^{\perp}}$. 
\end{proof}

By \cite{Ni}, 
$\widetilde{NS}(S, B)$  admits an embedding of $U \oplus U$ 
if ${\rm rk\/}(NS(S)) \geq 13$.

\begin{corollary}\label{2-elementary}
When the lattice $NS(S)$ is $2$-elementary, 
i.e, $D_{NS(S)}\simeq ({\Z}/2{\Z})^{a}$, 
then  
$\# {\rm FM\/}^{d}(S) = 0$ 
unless 
$d=1, 2$. 
If $d=1, 2$, 
we have 
\[
\# {\rm FM\/}^{d}(S) = \# \left( O_{Hodge}(T(S)) \backslash I^{d}(D_{T(S)}) \right) .
\]
\end{corollary}

\begin{proof}
Because the lattice 
$M_{x}:=\langle \lambda (x), NS(S)\rangle $ 
is also $2$-elementary, 
$\mathcal{G}(M_{x})=\{ M_{x} \} $ and 
$\tau (x, M_{x})=1$ 
by \cite{Ni}. 
\end{proof}

\section{Twisted Fourier-Mukai partners of a $K3$ surface of Picard number $1$}

Let $S$ be a $K3$ surface with $NS(S)={\Z}H$, $(H, H)=2n$. 
In this section 
we shall describe all twisted FM partners of $S$ 
(up to isomorphism) 
as moduli spaces of stable sheaves on $S$ with explicit Mukai vectors, 
endowed with natural twistings. 
First of all, 
we calculate the twisted FM numbers by Theorem \ref{the formula}.

\begin{prop}\label{(d, n) number}
For a natural number $d$, 
we have 
$\# {\rm FM\/}^{d}(S) = 0$ 
unless 
$d^{2}|n$. 
If $d^{2}|n$, then 
\[
\# {\rm FM\/}^{d}(S) = \left\{
\begin{array}{rl}
\varphi (d)\cdot 2^{\tau (d^{-2}n)-2} & \: \: \: d\geq 3, d^{2}=n, \\ 
\varphi (d)\cdot 2^{\tau (d^{-2}n)-1} & \: \: \: d=1, 2 \; {\rm or\/}\; d^{2}<n.  
\end{array} 
\right. 
\]
Here 
$\varphi $ is the Euler function 
and 
$\tau (m)$ is the number of the prime factors of $m$. 
We define 
$\tau (1)=1$. 
\end{prop}

\begin{proof}
We have 
$D_{NS(S)}=\langle \frac{H}{2n}\rangle \simeq {\Z}/2n{\Z}$ 
and 
$O_{Hodge}(T(S))=\{ \pm {\rm id\/} \}$ 
(cf. \cite{H-L-O-Y}). 
Since 
\[
I^{d}(D_{NS(S)})\ne \phi \: \: \Leftrightarrow \: \: 
(\frac{H}{d}, \frac{H}{d})=\frac{2n}{d^{2}} \in 2{\Z} \: \: \Leftrightarrow \: \: 
d^{2}|n, 
\]
it follows that 
$\# {\rm FM\/}^{d}(S) = 0$ 
unless 
$d^{2}|n$.

Assume that $d^{2}|n$. 
Then we have 
$I^{d}(D_{NS(S)})=\{ k\frac{H}{d}, \: k\in ({\Z}/d{\Z})^{\times} \}$  
so that 
\[
\# ( O_{Hodge}(T(S)) \backslash I^{d}(D_{NS(S)}) ) = 
\left\{ 
\begin{array}{cl}
1 & \: d=1, 2,  \\ 
\frac{1}{2}\varphi (d) & \: d\geq 3.  
\end{array}
\right. 
\]

For 
$d\geq 3$ and $k \in ({\Z}/d{\Z})^{\times}$, 
we have 
$\langle {\Z}H, k\frac{H}{d} \rangle = {\Z}\frac{H}{d}$ 
and 
\[
\mathcal{G}\Bigl( {\Z}\frac{H}{d}\Bigr) = 
\Bigl\{ {\Z}\frac{H}{d} \Bigr\} = 
\left\{ 
\begin{array}{ll}
\mathcal{G}_{1}({\Z}\frac{H}{d}) & \: d^{2}=n,  \\ 
\mathcal{G}_{2}({\Z}\frac{H}{d}) & \: d^{2}<n .  
\end{array}
\right. 
\] 
Because  
$O_{Hodge}(T_{k\frac{H}{d}}, \alpha _{k\frac{H}{d}}) = \{ {\rm id\/} \}$,  
$O({\Z}\frac{H}{d}) = \{ \pm {\rm id\/} \}$ 
and 
\begin{equation}\label{eqn:isometry of disc form}
O( D_{{\Z}\frac{H}{d}} ) \simeq 
\left\{ 
\begin{array}{ll}
\{ {\rm id\/} \} & \: d^{2}=n,  \\ 
({\Z}/2{\Z})^{\tau (d^{-2}n)} & \: d^{2}<n ,   
\end{array}
\right. 
\end{equation}  
it follows from Theorem \ref{the formula} that 
\[
\# {\rm FM\/}^{d}(S) = \left\{
\begin{array}{rl}
\frac{1}{2}\varphi (d) \cdot 1 \cdot 2^{\tau (d^{-2}n)-1} & \: \: \:  d^{2}=n, \\ 
\frac{1}{2}\varphi (d) \cdot 2 \cdot 2^{\tau (d^{-2}n)-1} & \: \: \:  d^{2}<n.  
\end{array} 
\right. 
\]

Let $d \leq 2$. 
Since   
$O_{Hodge}(T_{\frac{H}{d}}, \alpha _{\frac{H}{d}}) = 
O({\Z}\frac{H}{d}) = \{ \pm {\rm id\/} \}$ 
and the isomorphisms (\ref{eqn:isometry of disc form}) still hold, 
we have  
$
\# {\rm FM\/}^{d}(S) =  
2^{\tau (d^{-2}n)-1} =
\varphi (d)\cdot 2^{\tau (d^{-2}n)-1}
$ 
by Theorem \ref{the formula}. 
\end{proof}

Now we shall construct the moduli spaces. 
Since the primitive embeddings of the lattice ${\Z}H$ 
into the $K3$ lattice $\Lambda _{K3}$ are unique modulo $O(\Lambda _{K3})$, 
we fix an isometry 
$H^{2}(S, {\Z}) \simeq \Lambda _{K3} = U^{3}\oplus E_{8}^{2}$ 
so that 
$H$ is written as $H=nu+v$, 
where 
$\{ u, v\} $ is a standard basis of the first $U$. 
Then 
the transcendental lattice $T(S)$ is written as 
$T(S)={\Z}(nu-v)\oplus U^{2}\oplus E_{8}^{2}$. 
Set 
$G:=nu-v \in T(S)$.

For 
a natural number $d$ with $d^{2}|n$, 
let 
$d^{-2}n = \mathop{\prod}_{i=1}^{\tau (d^{-2}n)} p_{i}^{e_{i}}$ 
be the prime decomposition. 
When $d^{2}=n$, 
we put $p_{1}=e_{1}=1$. 
Set 
\[
\Sigma := 
\Bigl\{ 
\Bigl. 
\{ i_{1}, \cdots , i_{a} \} \subset \{ 1, 2, \cdots , \tau (d^{-2}n) \} 
\: \Bigr| \:  
\{ i_{1}, \cdots , i_{a} \} \ne \phi , \: 
 \mathop{\prod}_{\{ i_{1}, \cdots , i_{a} \} } p_{i}^{2e_{i}} \geq \frac{n}{d^{2}} \: 
\Bigr\} . 
\]
We have $\# \Sigma = 2^{ \tau (d^{-2}n) -1}$. 
For 
$\sigma = \{ i_{1}, \cdots , i_{a} \} \in \Sigma $,  
set 
$r_{\sigma} := \mathop{\prod}_{i \in \sigma } p_{i}^{e_{i}}$
and 
$s_{\sigma} := r_{\sigma}^{-1}d^{-2}n$. 
Then 
$r_{\sigma}$ is coprime to $s_{\sigma}$. 
We have 
$r_{\sigma} > s_{\sigma}$ if $d^{2}<n$. 
Let 
$k \in ({\Z}/d{\Z})^{\times}$. 
By the arithmetic progression, 
we can find a natural number 
$\tilde{k} \in {\N}$ coprime to $2n$ with $\tilde{k} \equiv k \in {\Z}/d{\Z}$. 
For example, 
we may take $\tilde{k}=1$ if $k=\bar{1}$. 
Fix such $\tilde{k}$ for each $k \in ({\Z}/d{\Z})^{\times}$. 
For 
$(\sigma , k) \in \Sigma \times ({\Z}/d{\Z})^{\times}$, 
we define 
\begin{equation}\label{eqn:Mukai vector} 
v_{\sigma , k} := 
                 ( dr_{\sigma}, \tilde{k}H, \tilde{k}^{2}ds_{\sigma} ) \in \widetilde{NS}(S).  
\end{equation}
Then 
$v_{\sigma , k}$ is a primitive isotropic vector in $\widetilde{NS}(S)$ 
with 
${\rm div\/}(v_{\sigma , k} )=d$.

Let 
$M_{\sigma , k} := M_{H}(v_{\sigma , k})$ 
be the coarse moduli space of 
Gieseker stable sheaves on $S$ with Mukai vector $v_{\sigma , k}$. 
It follows from \cite{Mu} that 
$M_{\sigma , k}$ is non-empty and is a $K3$ surface. 
Let 
$\alpha _{\sigma , k} \in {\rm Br\/}(M_{\sigma , k})$ 
be the obstruction to the existence of  a universal sheaf (\cite{Ca2}) 
with a B-field lift $B_{\sigma , k} \in H^{2}(M_{\sigma , k}, {\Q})$. 
By \cite{H-S1}, 
there exists an orientation-preserving Hodge isometry 
\begin{equation}\label{eqn:Hodge isometry 1}
\widetilde{\Phi} : 
\widetilde{H}(M_{\sigma , k}, B_{\sigma , k}, {\Z}) 
\simeq 
\widetilde{H}(S, {\Z}) 
\: \: \: {\rm with\/} \: \: 
(0, 0, 1) \mapsto v_{\sigma , k} , 
\end{equation}
which 
induces a Hodge isometry 
\begin{equation}\label{eqn:Hodge isometry 2} 
H^{2}(M_{\sigma , k}, {\Z}) 
\simeq 
\Bigl( v_{\sigma , k}^{\perp}\cap \widetilde{H}(S, {\Z}) \Bigr) / {\Z}v_{\sigma , k} .
\end{equation}
(We remark that the Hodge isometry (\ref{eqn:Hodge isometry 2}) was proved first by \cite{Mu}.)

Using (\ref{eqn:Hodge isometry 2}), 
we shall calculate the isometry 
\[
\lambda _{\sigma , k} := 
\lambda _{H^{2}(M_{\sigma , k}, {\Z})} : 
(D_{NS(M_{\sigma , k})}, q) \simeq (D_{T(M_{\sigma , k})}, -q) . 
\]
Firstly, 
direct calculations 
show that 
\begin{eqnarray*}
     NS(M_{\sigma , k}) 
& \simeq  &
    \Bigl( v_{\sigma , k}^{\perp}\cap \widetilde{NS}(S) \Bigr) / {\Z}v_{\sigma , k} \\
& \simeq & 
  \left( v_{\sigma , k}^{\perp} \cap \Bigl\langle \frac{v_{\sigma , k}}{d}, \widetilde{NS}(S) \Bigr\rangle \right) 
   / {\Z}\frac{v_{\sigma , k}}{d} \\ 
&   =    & 
    \left( {\Z}\frac{v_{\sigma , k}}{d} \oplus {\Z}( 0, -\frac{H}{d}, -2s_{\sigma}\tilde{k} ) \right)  
    / {\Z}\frac{v_{\sigma , k}}{d} . 
\end{eqnarray*}
We define  
\[
H_{\sigma , k} := \Bigl[ ( 0, -\frac{H}{d}, -2s_{\sigma}\tilde{k} ) \Bigr] \: \: \in NS(M_{\sigma , k}).  
\]
We have 
$NS(M_{\sigma , k}) = {\Z}H_{\sigma , k}$ 
and 
$
\tilde{k}H_{\sigma , k} = [ ( r_{\sigma}, 0, -s_{\sigma}\tilde{k}^{2} ) ]  
$. 
Since 
the Hodge isometry (\ref{eqn:Hodge isometry 1}) 
is orientation-preserving,  
$H_{\sigma , k}$ 
is the ample generator of  $NS(M_{\sigma , k})$. 
For 
the transcendental lattice 
we have 
\begin{eqnarray*}
T(M_{\sigma , k}) 
& \simeq & 
\overline{T(S)\oplus {\Z}v_{\sigma , k}} / {\Z}v_{\sigma , k} \\
& = & 
\Bigl( 
\Bigl\langle {\Z}v_{\sigma , k} \oplus {\Z}G , \frac{v_{\sigma , k}+\tilde{k}G}{d} \Bigr\rangle
\oplus U^{2} \oplus E_{8}^{2} 
\Bigr) 
/ {\Z}v_{\sigma , k} .
\end{eqnarray*}
Here 
$\overline{L}$ means 
the primitive hull of the lattice $L$ 
in $\widetilde{H}(S, {\Z})$. 
Since 
$\lambda _{\widetilde{H}(S, {\Z})} (\frac{v_{\sigma , k}}{d}) = \frac{\tilde{k}G}{d} \in D_{T(S)}$, 
it follows from 
Lemma \ref{twist and disc form} and 
the Hodge isometry (\ref{eqn:Hodge isometry 1}) 
that 
the twisting $\alpha _{\sigma , k}$ is given by 
\begin{equation}\label{twisting}
\alpha _{\sigma , k} : 
T(M_{\sigma , k}) 
\stackrel{\widetilde{\Phi} \circ {\rm e\/}^{B_{\sigma , k}}}{\longrightarrow}
\Bigl\langle \frac{G}{d}, T(S) \Bigr\rangle / T(S) 
\simeq 
\Bigl\langle \frac{-\tilde{k}G}{d} \Bigr\rangle 
\simeq 
{\Z}/d{\Z} . 
\end{equation} 
Via 
the Hodge isometry (\ref{eqn:Hodge isometry 2}), 
we obtain a Hodge isometry 
\begin{equation}\label{eqn:Hodge isometry 3}
g_{\sigma , k} := 
\widetilde{\Phi} \circ {\rm e\/}^{B_{\sigma , k}} : 
T(M_{\sigma , k}) \simeq \Bigl\langle \frac{G}{d}, T(S) \Bigr\rangle , 
\: \: \: 
g_{\sigma , k} \Bigl( \Bigl[ \frac{\tilde{k}G+v_{\sigma , k}}{d} \Bigr] \Bigr) = \frac{\tilde{k}G}{d} . 
\end{equation}
Set 
$G_{\sigma , k} := g_{\sigma , k}^{-1} (\frac{G}{d})$, 
which is a primitive vector in 
$T(M_{\sigma , k})$. 
From  
(\ref{eqn:Hodge isometry 3}), 
for 
$(\sigma , k), (\sigma ', k') \in \Sigma \times ({\Z}/d{\Z})^{\times}$ 
we have  
a Hodge isometry 
\begin{equation}\label{eqn:Hodge isometry 4}
g_{\sigma , k}^{\sigma ', k'} : = 
g_{\sigma ', k'}^{-1} \circ g_{\sigma , k} : 
T(M_{\sigma , k}) \simeq T(M_{\sigma ', k'}),  
\: \: \:  \: \: 
g_{\sigma , k}^{\sigma ', k'} ( G_{\sigma , k} ) 
= G_{\sigma ', k'}. 
\end{equation}
Therefore  
we have an equivalence 
$D^{b}(M_{\sigma , k}) \simeq D^{b}(M_{\sigma ' , k'})$ 
by Proposition \ref{H-S}. 
Note that 
the Hodge isometries 
from $T(M_{\sigma , k})$ 
to  $T(M_{\sigma ', k'})$ 
are just 
$\{ \pm g_{\sigma , k}^{\sigma ', k'} \}$.

Now 
we calculate $\lambda _{\sigma , k}$ 
as follows. 
We have 
\begin{eqnarray*}
& &        D_{NS(M_{\sigma , k})} 
    \simeq 
         \Bigl\langle \frac{\tilde{k}H_{\sigma , k}}{2r_{\sigma}s_{\sigma}} \Bigr\rangle 
     =     
        \Bigl\langle ( \frac{1}{2s_{\sigma}}, 0, \frac{-\tilde{k}^{2}}{2r_{\sigma}} ) \Bigr\rangle , \\ 
& &        D_{T(M_{\sigma , k})} 
    \simeq 
        \Bigl\langle \frac{\tilde{k}G_{\sigma , k}}{2r_{\sigma}s_{\sigma}} \Bigr\rangle 
     =     
      \Bigl\langle ( \frac{1}{2s_{\sigma}}, \tilde{k}du, \frac{\tilde{k}^{2}}{2r_{\sigma}} ) \Bigr\rangle .
\end{eqnarray*}
Firstly 
assume that $r_{\sigma}s_{\sigma}$ is odd. 
With respect to orthogonal decompositions 
\[
D_{NS(M_{\sigma , k})} 
\simeq  
({\Z}/2{\Z}) \oplus \mathop{\bigoplus}_{i=1}^{\tau (r_{\sigma}s_{\sigma})} ({\Z}/p_{i}^{e_{i}}{\Z}), 
 \: \: \: \: \:  
D_{T(M_{\sigma , k})}  
\simeq  
({\Z}/2{\Z}) \oplus \mathop{\bigoplus}_{i=1}^{\tau (r_{\sigma}s_{\sigma})} ({\Z}/p_{i}^{e_{i}}{\Z}), 
\]
we write 
\begin{equation}
\frac{H_{\sigma , k}}{2r_{\sigma}s_{\sigma}} 
= ( \bar{1}, x_{\sigma , k}^{1}, \cdots , x_{\sigma , k}^{\tau (r_{\sigma}s_{\sigma})} ),  
\: \: \: \: \: 
\frac{G_{\sigma , k}}{2r_{\sigma}s_{\sigma}} 
= ( \bar{1}, y_{\sigma , k}^{1}, \cdots , y_{\sigma , k}^{\tau (r_{\sigma}s_{\sigma})} ). 
\end{equation}
Each 
$x_{\sigma , k}^{i}$, 
$y_{\sigma , k}^{i}$
are 
generators of ${\Z}/p_{i}^{e_{i}}{\Z}$. 
Since 
we have 
\begin{eqnarray*}
& & s_{\sigma}\frac{\tilde{k}H_{\sigma , k}}{2r_{\sigma}s_{\sigma}} + 
    s_{\sigma}\frac{\tilde{k}G_{\sigma , k}}{2r_{\sigma}s_{\sigma}} \; = \; 
                                                        [ (1, \tilde{k}ds_{\sigma}u, 0) ]
                                                \: \; \in  H^{2}(M_{\sigma , k}, {\Z}) , \\ 
& & - r_{\sigma}\frac{\tilde{k}H_{\sigma , k}}{2r_{\sigma}s_{\sigma}}  
    + r_{\sigma}\frac{\tilde{k}G_{\sigma , k}}{2r_{\sigma}s_{\sigma}} \; = \; 
                                                        [ (0, \tilde{k}dr_{\sigma}u, \tilde{k}^{2}) ]
                                                \: \; \in  H^{2}(M_{\sigma , k}, {\Z}) , 
\end{eqnarray*} 
the isometry 
$\lambda _{\sigma , k} = \lambda _{H^{2}(M_{\sigma , k}, {\Z})}$ 
is 
determined 
by 
\begin{equation}\label{eqn:patching} 
\lambda _{\sigma , k} ( x_{\sigma , k}^{i} ) = 
\left\{ 
\begin{array}{rl}
 y_{\sigma , k}^{i} & \: \: \:  i \in \sigma , \\ 
-y_{\sigma , k}^{i} & \: \: \:  i \notin \sigma .   
\end{array} 
\right. 
\end{equation}
If 
$r_{\sigma}s_{\sigma}$ is even, 
put $p_{1}=2$ and write 
\begin{eqnarray*}
\frac{H_{\sigma , k}}{2r_{\sigma}s_{\sigma}} 
= ( x_{\sigma , k}^{1}, \cdots , x_{\sigma , k}^{\tau (r_{\sigma}s_{\sigma})} ) 
& \in &
D_{NS(M_{\sigma , k})} 
\simeq  
({\Z}/2^{e_{1}+1}{\Z}) \oplus \mathop{\bigoplus}_{i=2}^{\tau (r_{\sigma}s_{\sigma})} ({\Z}/p_{i}^{e_{i}}{\Z}), \\ 
\frac{G_{\sigma , k}}{2r_{\sigma}s_{\sigma}} 
= ( y_{\sigma , k}^{1}, \cdots , y_{\sigma , k}^{\tau (r_{\sigma}s_{\sigma})} ) 
& \in & 
D_{T(M_{\sigma , k})} 
\simeq  
({\Z}/2^{e_{1}+1}{\Z}) \oplus \mathop{\bigoplus}_{i=2}^{\tau (r_{\sigma}s_{\sigma})} ({\Z}/p_{i}^{e_{i}}{\Z}). 
\end{eqnarray*}
Then 
$\lambda _{\sigma , k}$ is determined by 
the same equations as $(\ref{eqn:patching})$.

\begin{prop}\label{moduli(sigma, k)}
For the twisted $K3$ surfaces 
$(M_{\sigma , k}, \alpha _{\sigma , k})$ 
the followings hold. 

$(1)$  
The underlying $K3$ surfaces 
$\{ 
M_{\sigma , k} , \: 
(\sigma , k) \in \Sigma \times ({\Z}/d{\Z})^{\times} 
\}$
are derived equivalent to each other.  

$(2)$ 
We have 
$M_{\sigma , k} \simeq M_{\sigma ', k'}$ 
if and only if 
$\sigma = \sigma '$. 

$(3)$ 
We have 
$(M_{\sigma , k} , \alpha _{\sigma , k}) 
\simeq 
(M_{\sigma , 1} , k^{-1}\alpha _{\sigma , 1})$.  
\end{prop}

\begin{proof}
We already proved the assertion $(1)$  
using Mukai-Orlov's theorem.

Assume the existence of an effective Hodge isometry 
$\Phi : 
H^{2}(M_{\sigma , k}, {\Z}) 
\simeq 
H^{2}(M_{\sigma ', k'}, {\Z})$ 
for 
$(\sigma , k) , (\sigma ', k') \in \Sigma \times ({\Z}/d{\Z})^{\times}$. 
We must have 
$\Phi (H_{\sigma , k}) = H_{\sigma ', k'}$ 
and 
$\Phi (G_{\sigma , k}) = \pm G_{\sigma ', k'}$. 
Since 
the identity 
$
\lambda _{\sigma ', k'} \circ r(\Phi |_{NS(M_{\sigma , k})}) 
 = 
 r(\Phi |_{T(M_{\sigma , k})}) \circ \lambda _{\sigma , k} 
$ 
holds, 
it follows from  $(\ref{eqn:patching})$ 
that 
$\sigma = \sigma ' \; {\rm or\/} \; (\sigma ')^{c}$. 
By the definition of $\Sigma $ 
we have $\sigma = \sigma '$.

Conversely 
let $\sigma = \sigma '$. 
If we define 
$
\gamma _{\sigma , k}^{\sigma , 1} : 
NS(M_{\sigma , k}) \simeq NS(M_{\sigma , 1})
$ 
by
$
\gamma _{\sigma , k}^{\sigma , 1} (H_{\sigma , k}) 
= H_{\sigma , 1} 
$, 
then
$
\gamma _{\sigma , k}^{\sigma , 1} \oplus g_{\sigma , k}^{\sigma , 1}
$ 
extends to an effective Hodge isometry 
$
H^{2}(M_{\sigma , k}, {\Z}) 
\simeq 
H^{2}(M_{\sigma , 1}, {\Z})
$. 
Since  
the twistings 
$\alpha _{\sigma , k}$  
are given by (\ref{twisting}), 
we have 
$
( g_{\sigma , k}^{\sigma , 1} )^{\ast} \alpha _{\sigma , 1} 
 = k \alpha _{\sigma , k}
$, 
or equivalently, 
$
( g_{\sigma , k}^{\sigma , 1} )^{\ast} ( k^{-1}\alpha _{\sigma , 1} )
 = \alpha _{\sigma , k}
$.
\end{proof}

Finally, 
we obtain 
the following theorem. 
\begin{theorem}\label{(d, n) partners}
Let $S$ be a $K3$ surface 
with 
$NS(S)={\Z}H$, $(H, H)=2n$, 
and let $d$ be a natural number satisfying $d^{2}|n$. 

$(1)$ 
If $d^{2}<n$ or $d\leq 2$, 
then 
\begin{equation*} 
{\rm FM\/}^{d}(S) = \left\{ \left. ( M_{\sigma , k}, \alpha _{\sigma , k} ) \right| \: 
                            (\sigma , k) \in  \Sigma \times ({\Z}/d{\Z})^{\times} \right\} . 
\end{equation*}

$(2)$ 
Assume $d^{2}=n$ and $d\geq 3$. 
Choose a set 
$\{ j \} \subset ({\Z}/d{\Z})^{\times}$ 
of representatives of  
$({\Z}/d{\Z})^{\times} / \{ \pm {\rm id\/} \} $.  
then 
\begin{equation*} 
{\rm FM\/}^{d}(S) = \left\{ \left. ( M_{\sigma , k}, \alpha _{\sigma , k} ) \right| \: 
                            (\sigma , k) \in \Sigma \times \{ j \} \right\} . 
\end{equation*}
\end{theorem}

\begin{proof}
That 
$(M_{\sigma , k}, \alpha _{\sigma , k}) \in {\rm FM\/}^{d}(S)$ 
is a direct consequence of C\u ald\u araru's theorem (\cite{Ca2}). 
By Proposition \ref{(d, n) number}, 
it suffices to show that 
the twisted $K3$ surfaces in the right hand sides are not isomorphic to each other.

$(1)$ 
By Proposition \ref{moduli(sigma, k)}, 
the right hand side is identified with the set  
\[
\{ 
( M_{\sigma , 1}, k\alpha _{\sigma , 1} ) , \: 
(\sigma , k) \in \Sigma \times ({\Z}/d{\Z})^{\times} 
\} . 
\]
We have
$
( M_{\sigma , 1}, k\alpha _{\sigma , 1} ) 
\not\simeq 
( M_{\sigma ' , 1}, k'\alpha _{\sigma ' , 1} )
$
if 
$\sigma \ne \sigma '$. 
If 
$d^{2}<n$, 
then 
$Aut(M_{\sigma , 1}) =\{ {\rm id\/} \} $ 
so that 
$
( M_{\sigma , 1}, k\alpha _{\sigma , 1} ) 
\not\simeq 
( M_{\sigma , 1}, k'\alpha _{\sigma , 1} )$
if 
$k \ne k'$. 
If $d\leq 2$, 
there remains nothing to prove.

$(2)$ 
Let 
$d^{2}=n$ and $d\geq 3$. 
The right hand side is identified with the set  
\[
\{ 
( M_{\sigma , 1}, k\alpha _{\sigma , 1} ) , \: 
(\sigma , k) \in \Sigma \times \{ j \} 
\} . 
\]
Similarly as $(1)$, 
We have
$
( M_{\sigma , 1}, k\alpha _{\sigma , 1} ) 
\not\simeq 
( M_{\sigma ' , 1}, k'\alpha _{\sigma ' , 1} )
$
if 
$\sigma \ne \sigma '$. 
Now 
the $K3$ surface $M_{\sigma , 1}$ admits 
an anti-symplectic involution $\iota _{\sigma , 1}$ 
and 
$Aut(M_{\sigma , 1}) =\{ {\rm id\/} , \iota _{\sigma , 1} \} $. 
Therefore  
$
( M_{\sigma , 1}, k\alpha _{\sigma , 1} ) 
\not\simeq 
( M_{\sigma , 1}, k'\alpha _{\sigma , 1} )$
if 
$k \ne \pm k'$.
\end{proof}

When $d=1$ with $\tilde{k}=1$, 
the result is proved in \cite{H-L-O-Y'}.

\end{document}